\documentclass[12pt]{amsart}
\usepackage{mathrsfs, amssymb, array, upgreek}
\usepackage{verbatim}
\usepackage{mathcomp}
\usepackage{hyperref}
\usepackage[matrix,arrow,curve]{xy}
\usepackage{longtable}
\usepackage{enumitem}
\usepackage{upgreek,textcomp}
\makeatletter 
\@addtoreset{equation}{subsection} 
\makeatother
\usepackage{mathabx}

\newcommand{\OOO}{\mathscr{O}}

\newcommand{\CC}{\mathbb{C}}

\newcommand{\PP}{\mathbb{P}}
\newcommand{\QQ}{\mathbb{Q}}

\newcommand{\RR}{\mathbb{R}}
\newcommand{\bF}{\mathbf{F}}

\newcommand{\Sym}{\mathfrak{S}}
\newcommand{\Alt}{\mathfrak{A}}

\newcommand{\Sing}{\operatorname{Sing}}
\newcommand{\Bir}{\operatorname{Bir}}
\newcommand{\Pic}{\operatorname{Pic}}

\newcommand{\Cl}{\operatorname{Cl}}

\newcommand{\Aut}{\operatorname{Aut}}
\newcommand{\Cr}{\operatorname{Cr}}

\newcommand{\rk}{\operatorname{rk}}

\newcommand{\B}{\mathbf{B}}

\newcommand{\GL}{\operatorname{GL}}
\newcommand{\PSL}{\operatorname{PSL}}
\newcommand{\PGL}{\operatorname{PGL}}

\newcommand{\SL}{\operatorname{SL}}

\newcommand{\mumu}{{\boldsymbol{\mu}}}

\newcommand{\Syl}{\operatorname{Syl}}

\newcommand{\rr}{\operatorname{r}}
\newcommand{\p}{\operatorname{p}_{\mathrm{a}}}
\newcommand{\g}{\operatorname{g}}

\newcommand{\xref}[1]{\textup{\ref{#1}}}

\makeatletter
\@addtoreset{equation}{subsection}
\makeatother

\swapnumbers
\theoremstyle{plain}

\newtheorem{lemma}[subsection]{Lemma}

\newtheorem{proposition}[subsection]{Proposition}

\newtheorem{corollary}[subsection]{Corollary}

\theoremstyle{definition}
\newtheorem{notation}[subsection]{Notation}
\newtheorem{example}[subsection]{Example}

\renewcommand\labelenumi{\rm (\roman{enumi})}
\renewcommand\theenumi{\rm (\roman{enumi})}
\renewcommand\labelenumii{\rm (\alph{enumii})}
\renewcommand\theenumii{\rm (\alph{enumii})}

\newcounter{NN}
\numberwithin{NN}{section}
\def\nr{\refstepcounter{NN}{\theNN}}
\renewcommand{\theNN}{\rm\arabic{NN}${}^o$}

\title{Embeddings of the symmetric groups \\ 
to the space Cremona group}

\author{Yuri Prokhorov}

\thanks{This work is supported by the Russian Science Foundation under grant \textnumero 18-11-00121.}
\date{}

\address{Steklov Mathematical Institute of RAS,
8 Gubkina street, Moscow 119991, Russia.
\newline
HSE University, Russian Federation,
Laboratory of Algebraic Geometry, 6 Usacheva str., Moscow, 119048, Russia.
}
\email{prokhoro@mi-ras.ru}

\begin{document}

\maketitle
\tableofcontents

\section{Introduction}
The classification of finite subgroups in the Cremona groups $\Cr_n(\CC)$ is an old-standing problem that goes back to works of classics of Italian algebraic geometry. In the last two decades the interest to this problem was reactivated; see for example \cite{Dolgachev-Iskovskikh}, \cite{P:ECM}, and references therein. 
In particular, now there is a basically complete classification of finite subgroups in the plane Cremona group $\Cr_2(\CC)$ \cite{Dolgachev-Iskovskikh}.

In this paper we are interested in embeddings of symmetric groups $\Sym_N$
to Cremona group $\Cr_3(\CC)$ and, more generally, to groups of birational self-maps of three-dimensional rationally connected varieties. 
This problem is interesting not only in its own sake but also in relation
with computation of essential dimension of $\Sym_N$ (cf. \cite{Duncan2010}).

\begin{proposition}
\label{main0}
Let $X$ be a rationally connected threefold and let $\Bir(X)$
be the group of its birational self-maps.
\begin{enumerate}
\item 
\label{main0a}
For $n\ge 8$ the symmetric group $\Sym_n$ does not admit any embedding 
to $\Bir(X)$.
\item 
\label{main0b}
Any embedding $\Sym_7\subset \Bir(X)$ up to conjugation is induced by 
the action of $\Sym_7$ on the smooth variety
$X_6'\subset \PP^5\subset \PP^6$ given by the equations 
\begin{equation}
\label{eq:X6}
\sum_{i=1}^7 x_i=\sum_{i=1}^7 x_i^2=\sum_{i=1}^7 x_i^3=0
\end{equation} 
with natural action of $\Sym_7$ on $x_1,\dots,x_7$ by permutations.
Moreover, any three-dimensional $\Sym_7$-Mori fiber space over a rationally connected base is equivariantly isomorphic to the variety~\eqref{eq:X6}.
\end{enumerate}
In particular, $\Sym_n$ is non embeddable to $\Cr_3(\CC)$ for $n\ge 7$
\textup(because the variety \eqref{eq:X6} is not rational \cite{Beauville2012}\textup).
\end{proposition}
Note that embeddings to $\Cr_3(\CC)$ of some other classes of ``large'' 
finite groups were studied in \cite{P:JAG:simple}, \cite{BCDP:QS}.

Our second main result is related to the symmetric group $\Sym_6$. 
Unfortunately it is not complete.

\begin{proposition}
\label{main2}
Let $Y$ be a rationally connected threefold.
Then for any embedding $\Sym_6\subset \Bir(Y)$ there exists 
a $\Sym_6$-equivariant birational map $Y\dashrightarrow X$ such that one 
of the following holds:
\begin{enumerate}
\item 
\label{main2a}
$X$ is a Fano threefold with at worst terminal Gorenstein $\Sym_6\QQ$-factorial singularities and $\uprho(X)=1$;
\item 
\label{main2b}
$X$ is a Fano threefold with terminal $\Sym_6\QQ$-factorial singularities and $\uprho(X)=1$ such that 
all the non-Gorenstein points of $X$ are cyclic quotients of index $2$. 
The number $n$ of these points equals $12$ or $15$ and there are two possibilities:
\begin{enumerate}
\renewcommand\labelenumii{\rm (\Alph{enumii})}
\renewcommand\theenumii{\rm (\Alph{enumii})}
\item 
\label{cor-prop:RR12}
$n=12$, $-K_X^3=2g+4$, $g\ge -1$, $\dim |-K_X|=g+1$;
\item 
\label{cor-prop:RR15}
$n=15$, $-K_X^3=2g+11/2$, $g\ge -2$, $\dim |-K_X|=g+1$.
\end{enumerate}
\end{enumerate}
Moreover, any $G$-Mori fiber space is equivariantly isomorphic to one of the above cases.
\end{proposition}
Recall that a normal $G$-variety $X$ is said to be \textit{$G\QQ$-factorial} if 
some multiple $nD$ of any $G$-invariant divisor is Cartier \cite{P:G-MMP}.

We do not know any examples of non-Gorenstein Fano threefolds
admitting an $\Sym_6$-action. We expect that the case \ref{main2}\ref{main2b}
does not occur. In contrast there are a lot of examples of actions of $\Sym_6$ on
Gorenstein Fano threefolds (case \ref{main2}\ref{main2a}).
In Sect.~\ref{sect:ex} we collect known ones. 
However we do not assert that our collection is complete. 
The particular the case of actions of $\Sym_6$ on del Pezzo threefolds is 
completely studied in Sect.~\ref{section:del-Pezzo}.

\subsection*{Acknowledgements.} 
The author would like to thank the anonymous referees 
for their remarks that helped him to improve the presentation.

\section{Preliminaries}
We work over a complex number field $\CC$ throughout.
\begin{notation}
\begin{itemize}
\item 
$\Sym_n$ and $\Alt_n$ denote the symmetric and the alternating groups, respectively. 
\item 
As usual, $\Pic(X)$ denotes the Picard group of a variety $X$ and 
$\uprho(X)$ is the rank of $\Pic(X)$.
\item 
$\Cl(X)$ denotes the Weil divisor class group of a normal variety $X$ and 
$\rr(X):=\rk \Cl(X)$.
\item
If a group $G$ acts on an object $A$, then $A^G$ is the set of $G$-invariant elements.
\item
If a group $G$ acts on a variety $X$, then 
$\uprho(X)^G:=\rk \Pic(X)^G$ and $\rr(X)^G:=\rk \Cl(X)^G$.

\end{itemize}
\end{notation}
Throughout this paper we use the terminology and notation of the equivariant minimal model program \cite{P:G-MMP}. In particular, a \textit{$G\QQ$-Fano variety} $X$ is a variety equipped with an action of a finite group $G$ such that $X$ has only terminal $G\QQ$-factorial singularities, $\rk\Pic(X)=1$, and the anticanonical divisor $-K_X$ is ample.
In this situation, we say that $X$ is a \textit{$G$-Fano variety} if $X$ is Gorenstein
or, equivalently, $K_X$ is a Cartier divisor. 

For any (possibly singular) Fano threefold $X$ we define its \textit{genus} as follows
\[
g=\g(X):=\dim |-K_X|-1=\dim H^0(X,-K_X)-2. 
\]
Thus $\g(X)$ is an integer $\ge -2$. This definition agrees with usual definition of genus for smooth Fano threefolds \cite{IP99}.

We need some standard information about groups $\Sym_6$ and $\Alt_6$ 
and their actions on finite sets and lower-dimensional algebraic varieties.

\begin{lemma}
\label{lem:rep}
\begin{enumerate}
 \item 
Let $\rho$ be a faithful irreducible representation of the group $\Sym_6$.
Then $\dim(\rho)\in \{5,\, 9,\, 10,\, 16\}$.
 \item 
Let $\psi$ be a faithful irreducible representation of the group $\Alt_6$.
Then $\dim(\psi)\in \{5,\, 8,\, 9,\, 10\}$. 
\end{enumerate}
\end{lemma}

Recall that $[\Aut(\Sym_6):\operatorname{Inn}(\Sym_6)]=2$,
where $\operatorname{Inn}(\Sym_6)$ is the subgroup of inner automorphisms.
Let $\upsilon$ be an outer automorphism.

\begin{proposition}[see e.g. {\cite[\S~2.6, Theorem~2.4]{Wilson2009}}, \cite{atlas}]
\label{prop:atlas}
$\ $
\begin{enumerate}
 \item \label{prop:atlasS6}
Up to conjugacy, a maximal subgroup of $\Sym_6$ is one of the following: 
\newcommand{\hsh}{\hspace{0.6em}}
\[
\Alt_6,\hsh \Sym_5,\hsh 
\upsilon(\Sym_5),\hsh 
\mathrm{N}(\Sym_3\times \Sym_3)\simeq (\Sym_3\times \Sym_3)\rtimes \mumu_2,\hsh 
\Sym_4\times \Sym_2,\hsh \upsilon(\Sym_4\times \Sym_2).
\]
 \item \label{prop:atlasA7}
Up to conjugacy, a maximal subgroup of $\Alt_7$ is one of the following: 
\renewcommand{\hsh}{\hspace{2em}}
\[
\Alt_6,\hsh \Sym_5,\hsh H_1,\hsh H_2,\hsh \mathrm{N}(\Alt_4\times \Alt_3)
\]
where $H_1\simeq H_2\simeq \PSL_2(\bF_7)$ and $|\mathrm{N}(\Alt_4\times \Alt_3|=72$.
\end{enumerate}
\end{proposition}

\begin{corollary}
\label{cor:f-p}
Let $G:=\Sym_6$ act transitively on a set $\Omega$ with 
$|\Omega|\le 16$. Let $G_P$ be the stabilizer of $P\in \Omega$. 
There are only the following cases:\renewcommand\arraystretch{1.5}
\begin{center}\rm 
\begin{tabular}{|p{0.1\textwidth}||p{0.05\textwidth}|p{0.1\textwidth}|p{0.1\textwidth}|p{0.09\textwidth}|p{0.2\textwidth}|p{0.15\textwidth}|}
\hline
No.&
\nr\label{cor:f-p1}&
\nr \label{cor:f-p2}&
\nr \label{cor:f-p6}&
\nr \label{cor:f-p12}&
\nr \label{cor:f-p10}&
\nr \label{cor:f-p15}
\\
\hline
$|\Omega|$&$1$&$2$&$6$&$12$&$10$&$15$
\\
\hline
$G_P$& $\Sym_6$
&$\Alt_6$&$\Sym_5$&$\Alt_5$&
$(\Sym_3\times \Sym_3)\rtimes \mumu_2$&$\Sym_4\times \Sym_2$
\\
\hline
\end{tabular}
\end{center}
\end{corollary}

\begin{lemma}
\label{rat-surf}
The groups $\Sym_6$ and $\Alt_7$ do not admit embeddings into $\Cr_2(\CC)$
nor into $\Bir(S)$, where $S$ is an elliptic ruled surface.
\end{lemma}

\begin{proof}
See e.g. \cite{Dolgachev-Iskovskikh}.
\end{proof}

\begin{lemma}[\cite{Dolgachev-Iskovskikh}] 
\label{rat-surfA6}
\begin{enumerate}
\item 
The group $\Alt_6$ does not admit embeddings into 
$\Bir(S)$, where $S$ is an elliptic ruled surface.
\item 
Up to conjugacy 
and automorphisms of $\Alt_6$ there exists exactly 
one embedding $\Alt_6\subset \PGL_3(\CC)$.
\end{enumerate}
\end{lemma}
The image of $\Alt_6\hookrightarrow \PGL_3(\CC)$ whose image is called the \textit{Valentiner group} (see \cite{Dolgachev-Iskovskikh} and references therein). 

\begin{lemma}[{\cite[Lemma B.2]{Prokhorov2014x}}]
\label{lemma-fixed-point}
Let $X$ be a threefold with at worst terminal singularities
such that $\Aut(X)$ has a subgroup $G$ isomorphic to $\Alt_6$. Then $X$ contains no $G$-invariant points.
\end{lemma}

\section{Main reduction}
\label{main-reduction}

The following assertion is well known (see e.g. \cite[\S~14]{P:G-MMP})
\begin{proposition}
\label{thm:regularization}
Let $Y$ be a rationally connected variety and let $G\subset \Bir(Y)$ be a finite subgroup.
Then there there exists 
a $G$-equivariant birational map $Y\dashrightarrow X$, where $X$ is a projective variety having a $G$-Mori fiber space structure $f: X\to Z$.
\end{proposition}

\begin{proposition}[cf. {\cite[\S~4.2]{P:JAG:simple}}]
\label{thm:MFS}
Let $G=\Sym_{N}$ with $N\ge 6$ and let $f: X\to Z$ be a $G$-Mori fiber space,
where $X$ is rationally connected threefold.
Then $Z$ is a point, i.e. 
$X$ is a $G\QQ$-Fano threefold.
\end{proposition}

\begin{proof}
Assume that $\dim(Z)\ge 1$. Then $f$ is either 
$G\QQ$-del Pezzo fibration and $Z\simeq \PP^1$ or a $G\QQ$-conic bundle
and $Z$ a rational surface (see \cite[\S~10]{P:G-MMP}).
The map $f$ induces a homomorphism 
\[
\Phi: G \longrightarrow \Aut(Z). 
\]
Assume that $\dim(Z)=2$. Then the generic fiber $X_\eta$ of $f$ is a smooth rational curve. 
The kernel of $\Phi$ acts on $X_\eta$ faithfully.
Hence $\ker(\Phi)\not\supset \Alt_6$ and so $\Phi$ is injective. This contradicts Lemma \ref{rat-surf}. 

Thus we may assume that $f$ is a $G\QQ$-del Pezzo fibration.
Again by Lemma~\ref{rat-surf}
the image of $\Phi$ is either trivial or 
a cyclic group of order two. In both cases $\Phi(G)$ has a fixed point $o\in Z$.
Let $F$ be the scheme fiber over $o$.
Then the proof of \cite[Lemma B.5]{Prokhorov2014x} 
works without any changes. Thus $F\simeq \PP^2$.
Since $\Sym_N$, for $N\ge 6$, does not act faithfully on $\PP^2$,
the alternating subgroup $\Alt_N$ acts on $F$ trivially.
This again contradicts Lemma \ref{lemma-fixed-point}.
\end{proof}

\begin{corollary}
\label{cor:MFS}
Let $Y$ be a rationally connected algebraic threefold. Suppose that $\Bir(Y)$ contains a subgroup $G\simeq \Sym_{N}$ with $N\ge 6$.
Then there exists 
a $G$-equivariant birational map $Y\dashrightarrow X$, where $X$ is a $G\QQ$-Fano threefold.
\end{corollary}

\section{Non-Gorenstein Fano threefolds}

The goal of this section is to prove the following result.
\begin{proposition}[cf. {\cite[Lemma~6.1]{P:JAG:simple}}]
\label{prop:NG-point}
Let
$X$ be a non-Gorenstein Fano threefold 
with terminal singularities such that $\Aut(X)\supset G\simeq \Sym_6$.
Let $\Omega$ be the set of all non-Gorenstein points.
Then any point 
$P\in \Omega$ is a cyclic quotient singularity of type 
$\frac12(1,1,1)$, 
the action of $G$ on $\Omega$ is transitive, and one of the following holds:
\begin{enumerate}
\item\label{prop:NG-point12}
$|\Omega|=12$, $G_P\simeq \Alt_5$,

\item\label{prop:NG-point15}
$|\Omega|=15$, $G_P\simeq \Sym_4\times \Sym_2$.
\end{enumerate}
\end{proposition}

\begin{proof}
Fix a non-Gorenstein point $P\in X$ 
and let $P_1=P,\dots, P_{k}$ be its orbit.
Let $r\ge 2$ be the~Gorenstein index of $(X\ni P)$, let
\[
\pi\colon (X^{\sharp}, P^{\sharp})\longrightarrow (X,P)
\]
be the~index one
cover (see~\cite[Sect.~3.5]{Reid:YPG}), where $P^{\sharp}=
\pi^{-1}(P)$. Then $\pi$ 
is the~ topological universal cover (of degree $r$). Thus, there
is an exact~sequence of groups
\begin{equation}
\label{eq-seq}
1
\longrightarrow
M
\stackrel{\alpha}\longrightarrow
G_P^\sharp
\stackrel{\beta}\longrightarrow
G_P
\longrightarrow
1,
\end{equation} 
where $M\simeq \mumu_r$ and $G_P^{\sharp}$ is a~finite subgroup in
$\mathrm{Aut}(X^{\sharp}, P^{\sharp})$. 
The group $G_P^{\sharp}$ faithfully acts on the~Zariski tangent
space $T:=T_{P^{\sharp}, X^{\sharp}}$ to $X^{\sharp}$ at the~point
$P^{\sharp}$.
Recall that $(X^{\sharp}, P^{\sharp})$ is a~hypersurface
singularity \cite{Reid:YPG}. Hence, we have
$\dim(T)\le 4$. By the~classification of three-dimensional terminal singularities
(see~\cite[Sect.~6.1]{Reid:YPG}) the~action of the~group
$\alpha(M)$ on $T$ in some
coordinate system has one of the~following forms:
\begin{enumerate}
\renewcommand\labelenumi{\rm (\alph{enumi})}
\renewcommand\theenumi{\rm (\alph{enumi})}
\item 
\label{cases:term3}
$\dim(T)=3$, $(x_1,x_2,x_3)\longmapsto (\zeta_r x_1,\zeta_r^{-1} x_2,\zeta_r^{a} x_3)$,
\item 
\label{cases:term4}
$\dim(T)=4$, $(x_1,x_2,x_3,x_4)\longmapsto (\zeta_r x_1,\zeta_r^{-1} x_2,\zeta_r^{a} x_3, x_4)$,
\item 
\label{cases:term4e}
$\dim(T)=4$, $(x_1,x_2,x_3, x_4)\longmapsto (\zeta_4 x_1,-\zeta_4 x_2, \zeta_4 x_3, - x_4)$, $r=4$,
\end{enumerate}
where $\zeta_r$ is a~primitive $r$-th root of unity and $\gcd
(r,a)=1$. 

Denote by $T'\subset T$ the  subspace generated by the 
$M$-eigenspaces on which $M$ acts faithfully.
Thus $\dim(T')=3$ and $T'=T$ in the case~\ref{cases:term3}. The subspace $T'$ is $G_P^\sharp$-invariant and so
$G_P^\sharp\subset \GL(T')=\GL_3(\CC)$.

Recall that we can associate with $(X\ni P)$ a basket $\B(X,P)$, that is, a finite collection of cyclic quotient terminal
singularities $(X_\alpha\ni P_\alpha)$ \cite[Sect.~6.1 and Theorem 9.1 (III)]{Reid:YPG}.
Moreover, $\B(X,P)=\{(X\ni P)\}$ if and only if $(X\ni P)$ is a cyclic quotient
singularity (case~\ref{cases:term3}). In the case~\ref{cases:term4} all the singularities in $\B(X,P)$ are of index $r$.
In the case~\ref{cases:term4e} the basket $\B(X,P)$ contains a point of index $4$ and at least one point of index $2$.
Denote 
\[
\B(X):= \bigcup_{P\in X} \B(X,P). 
\]
According to  \cite{KMMT-2000} 
we have 
\begin{equation}
\label{eq:Kb}
\sum_{Q\in \B(X)} \left(r_Q-\frac 1{r_Q}\right)\le 24. 
\end{equation} 
In particular, $k\le |\Omega|\le 16$. 

We use Corollary \ref{cor:f-p}.
First we consider the case \ref{cor:f-p}.\ref{cor:f-p2}, i.e. $G\simeq \Alt_6$.
Since $\Alt_6$ 
has no non-trivial three-dimensional representations of dimension $\le 4$
(see Lemma~\ref{lem:rep}), 
the sequence \eqref{eq-seq} does not split. 
In particular, the representation $G_P^\sharp\hookrightarrow \GL(T')$ is irreducible (otherwise the restriction to a two-dimensional $G_P^\sharp$-invariant subspace $T_2\subset T'$ would be a faithful representation of $G_P^\sharp$).
Then $M$ acts on $T'$ by scalar multiplications and so $r=2$.
But in this case the determinant map $\det: G_P^\sharp\to \CC^*$ splits the sequence \eqref{eq-seq}.
Cases \ref{cor:f-p}.\ref{cor:f-p1} and \ref{cor:f-p6} are similar. 
Thus $k\ge 10$ and we may assume that any $G$-orbit on $\Omega$ has length 
at least $10$.

Then $r=2$ by \eqref{eq:Kb}. 
Moreover, $P\in X$ is a cyclic quotient singularity and $|\Omega|=k$, i.e. $P=P_1,\dots,P_{k}$
are all non-Gorenstein points of $X$.
We have $T':=T\simeq \CC^3$.
Since any normal subgroup of order $2$ is contained in 
the center, $G_P^\sharp$ is a central extension of $G_P$.
Let $H^\sharp=G_P^\sharp\cap \SL(T')$. Then $H^\sharp$ is a normal subgroup 
of $G_P^\sharp$ and $G_P^\sharp/H^\sharp\simeq \mumu_m$ for some $m$.
Moreover, $H^\sharp$ does not contain $M$
(because $H^\sharp\subset \SL(T')$ and $M$ is a group of order $2$
acting by scalar multiplication). Hence, $m$ is even.
Let $H:=\beta(H^\sharp)$. Then $H\simeq H^\sharp$ and $G_P/H\simeq \mumu_{m/2}$.

Consider the case \ref{cor:f-p}.\ref{cor:f-p10}. Then 
$G_P\simeq (\Sym_3\times \Sym_3) \rtimes \mumu_2$. The commutator subgroup $[G_P,G_P]$ is a group of order $18$ and 
\[
G_P/[G_P,G_P]\simeq \mumu_2\times \mumu_2.
\]
Since $G_P/H$ is cyclic, $H$ either coincides with $G_P$ or $[G_P:H]=2$.
If $H=G_P$, then $G_P^\sharp=H^\sharp \times M\simeq G_P \times \mumu_2$.
On the other hand, $G_P$ has no three-dimensional faithful representation.
Hence, $H$ is an index $2$ subgroup in $G_P$ and $G_P^\sharp/H^\sharp\simeq \mumu_4$.
Note that the center of $H\simeq H^\sharp$ is trivial
and $H$ has no faithful two-dimensional representations.
Hence, the representation of $H$ on $T'$ is irreducible.
Let $\Syl_3(H)$ be the Sylow $3$-subgroup. 
Then $\Syl_3(H)$ is normal in $H$ and $H/\Syl_3(H)$ transitively
permute $\Syl_3(H)$-eigenspaces. This is impossible.
\end{proof}

\begin{corollary}
\label{cor:NG-point7}
Let
$X$ be a Fano threefold 
with terminal singularities such that $\Aut(X)\supset G\simeq \Sym_7$.
Then $X$ is Gorenstein.
\end{corollary}
\begin{proof}
Easily follows from Proposition~\ref{prop:NG-point}.
We have to note only that $\Sym_7$ contains no subgroups
of index $12$ or $15$, hence $\Sym_7$ cannot act on the set 
$\Omega$ as in Proposition~\ref{prop:NG-point}.
\end{proof}

\begin{proof}[Proof of Proposition~\xref{main2}]
Let $G=\Sym_6$ be a subgroup in $\Bir(Y)$, where $Y$ is a rationally connected threefold.
By Proposition~\ref{thm:regularization} there exists an equivariant birational map $Y\dashrightarrow X$, where $X$ is a variety having a $G$-Mori fiber space structure $X\to Z$.
By Proposition~\ref{thm:MFS} the base $Z$ is a point, i.e. $X$ is a $G\QQ$-Fano threefold. Assume that $X$ is not Gorenstein.
By Proposition~\ref{prop:NG-point} we have only one of the cases
\ref{prop:NG-point}\ref{prop:NG-point12} or \ref{prop:NG-point}\ref{prop:NG-point15}. It remains to prove the last assertion of \ref{main2}\ref{main2b}.
For this, we apply 
the orbifold Riemann-Roch \cite{Reid:YPG}: 
\[
g+1=\dim |-K_X|= \frac12(-K_X)^3 -\frac14k + 2.
\]
Hence, $-K_X^3=2g-2+k/2$, where $k=12$ or $15$.
\end{proof}

\section{Proof of Proposition~\ref{main0}}
In this section we prove Proposition~\ref{main0}.
Note that the proof also follows from \cite[Theorem~4.3]{BCDP:QS}.
Let $X$ be a rationally connected variety such that $\Bir(X)\supset G\simeq \Sym_n$, $n\ge 7$. Let $G_0:=\Alt_n\subset G$.
By \cite[Theorem~1.5]{P:JAG:simple} we have $n=7$.
By Proposition~\ref{thm:regularization} we may assume that $X$ is a variety having a $G$-Mori fiber space structure $X\to Z$.
By Proposition~\ref{thm:MFS} the base $Z$ is a point, i.e. $X$ is a $G\QQ$-Fano threefold and by Corollary~\ref{cor:NG-point7} the singularities of $X$ are Gorenstein.

\begin{lemma}
\label{lemma:X2g-2}
The linear system $|-K_X|$
is very ample and defines an embedding 
\[
X=X_{2g-2}\subset \PP\left( H^0(X,-K_X)^\vee\right)=\PP^{g+1}.
\]
If $g\ge 5$, then the image $X_{2g-2}$ is an intersection 
of quadrics in $\PP^{g+1}$.
\end{lemma}
\begin{proof}
See \cite[Lemmas~5.3 and 5.4]{P:JAG:simple} 
\end{proof}

\begin{lemma}
If $g\le 4$, then $g=4$ and we have the case \xref{main0}\xref{main0b}.
\end{lemma}
\begin{proof}
The embedding $X=X_{2g-2}\subset \PP^{g+1}$ is $G$-equivariant and $\PP^{g+1}$ is naturally identified with $\PP(H^0(X,-K_X)^\vee)$.
Since $\Sym_7$ has no non-trivial representations of dimension $\le 5$,
we have $g\ge 4$. If $g=4$, then by \cite[Lemma~5.4]{P:JAG:simple} $X=X_6 \subset \PP^5$ is an intersection of a quadric and a cubic. Here the representation of $G$ on $H^0(X,-K_X)$ is irreducible and there exists exactly one invariant quadric. We obtain the case \xref{main0}\xref{main0b}. This variety $X_6'$ is not rational according to \cite{Beauville2012}.
\end{proof}

Assume that $g = 5$. 
In this case $\dim H^0(X,-K_X)=7$.
The variety $X \subset \PP^6$ is a complete intersection of three quadrics, say $Q_1,\, Q_2,\, Q_3$ (see Lemma~\ref{lemma:X2g-2}). 
They generate a two-dimensional linear system 
(net) $\mathscr{Q}$. 
The induced action of $G_0=\Alt_7$ on $\mathscr{Q}$ is trivial by Lemma~\ref{rat-surf}.
Therefore, all the quadrics in the net $\mathscr{Q}$ are $G_0$-invariant.
The faithful  $G_0$-representation $H^0(X,-K_X)^\vee$
is reducible: 
\[
H^0(X,-K_X)^\vee=W_1\oplus W_6, 
\]
where $\dim (W_1)=1$ and $\dim (W_6)=6$. 
Here $W_6$ is again a (unique) standard irreducible representation
having exactly one invariant quadric, say $Q'$.
On the other hand, 
the surface $\PP(W_6)\cap X$ is a complete intersection of 
(invariant) quadrics $Q_i\cap \PP(W_6)$.
Again this gives a contradiction.

From now on we assume that $g\ge 6$.
Let $G_0:=\Alt_7\subset G$. If $\rr(X)^{G_0}=1$, then by
\cite[Theorem~1.5]{P:JAG:simple} we have $X\simeq \PP^3$. In this case 
the action of $\Sym_7$ on $\PP^3$ lifts to a faithful action of a central extension $\tilde\Sym_7$
of $\Sym_7$ on $H^0(\PP^3,\, \OOO_{\PP^3}(1))$.
Since the Schur multiplier of $\Sym_7$ has order $2$ (see e.g. \cite{atlas}),
we may assume that 
$\tilde\Sym_7$ is either $\Sym_7$ itself or a double cover 
of $\Sym_7$. However such $\tilde\Sym_7$ has no faithful four-dimensional representations, a contradiction.

Thus we assume that $\rr(X)^{G_0}>1$. 
We claim that $X$ contains no planes. 
Indeed, in the case $\uprho(X)=1$ this is the statement of \cite[Theorem 1.1]{P:planes}.
For general case, we just note that the proof of \cite[Theorem 
1.1]{P:planes} works without changes. It uses only that $X$ is a $G$-Fano 
threefold with $g\ge 6$ such that $-K_X$ is very ample and the image 
of the anticanonical map is an intersection of quadrics 
(Lemma~\ref{lemma:X2g-2}). Thus $X$ contains no planes. 

Let 
$X_1\to X$ be a small $G_0\QQ$-factorialization
(we take $X_1=X$ if $X$ is $G_0\QQ$-factorial). Then $\uprho(X_1)^{G_0}>1$. 
Run $G_0$-equivariant MMP on~$X_1$:
\[ 
\xymatrix{
X_1\ar[r]^{\varphi_{1}}&X_2\ar[r]^{\varphi_{2}}&\cdots\ar[r]^{\varphi_{N-2}}&X_{N-1}\ar[r]^{\varphi_{N-1}}&X_N
}
\]
Since $X$ contains no planes, the variety $X_1$ contains no surfaces $S_1$ 
such that $(-K_{X_1})^2\cdot S_1=1$. 
Then by the classification of $G$-extremal contractions (see e.g. 
\cite[Theorems~7.1.1 and 8.2.4]{P:G-MMP}) the variety $X_2$ 
has at worst terminal Gorenstein singularities whose 
anticanonical class $-K_{X_2}$ is nef, big, and does not contract divisors.
Moreover, $\dim |-K_{X_2}|\ge \dim |-K_{X_1}|$ and $X_2$ contains no 
$G_0$-invariant effective divisors $S_2$ such that 
$(-K_{X_2})^2\cdot S_2=1$. Continuing the process we end up with
a $G_0$-Mori fiber space $X_N/Z$ with 
\[
\dim |-K_{X_N}|\ge \dim |-K_{X}|=g+1\ge 7.
\]
Then by \cite[Theorem~1.5]{P:JAG:simple} we have $X_N\simeq \PP^3$.
Denote $Y:=X_{N-1}$ and consider the last contraction 
\[
\varphi=\varphi_{N-1}: Y=X_{N-1} \longrightarrow X_N=\PP^3.
\]
Let $E\subset Y$ be its exceptional divisor. 
Assume that $\dim\varphi(E)=0$. 
Again by the classification of extremal contractions 
\cite[Theorems~7.1.1 and 8.2.4]{P:G-MMP} $\varphi$ is the blowup of a $G_0$-invariant set of distinct points $P_1,\dots,P_r$. 
By Lemma~\ref{lemma-fixed-point} the stabilizer of 
any point on $\PP^3$ is not isomorphic to $\Alt_6$. Hence, $r\ge 8$
and so 
$(-K_Y)^3= (-K_{\PP^3})^3 -8r\le 0$
a contradiction.
Therefore, $C:=\varphi(E)$ is a curve. Clearly, $C$ is not contained in a plane.
Assume that $C$ is irreducible. 
Let $C'\to C$ be the normalization. 
Since $\Aut(\PP^1)$ contains no subgroups isomorphic to $\Alt_7$, $C'$ is not 
rational. Also, one can see that $C'$ is not
an elliptic curve. We have
\begin{equation*}
(-K_{X_N})^3 - (-K_{X_{N-1}})^3 = 2(-K_{X_{N-1}})^2 \cdot E + 2\p(C)-2. 
\end{equation*}
(see e.g. \cite[Proposition~5.1]{P:planes}).

Thus $2\p(C)-2< 64-10$ and $\p(C)\le 27$. 
On the other hand, by the Hurwitz bound
the the order of the automorphism group of $C'$ is at most $84(\g(C')-1)\le 2184$.
Since $2184<|\Alt_7|$, the group $G_0=\Alt_7$ cannot act effectively on $C$, a contradiction.

Thus, the curve $C$ is reducible.
Let $C_1,\dots,C_r$ be its irreducible components
and let $G_1$ be the stabilizer of $C_1$ in $G_0=\Alt_7$. 
By \cite[Lemma~5.3]{P:JAG:simple} the linear system $|-K_{X_{N-1}}|$
is base point free.
Let $S_1,S_2\in |-K_{X_{N-1}}|$ be two general members. Then $C\subset \varphi(S_1)\cap \varphi(S_2)$, where $\varphi(S_i)\in |-K_{\PP^3}|$. Hence 
$r\le \deg C\le 16$.
Since the group $G_0$ permutes $C_1,\dots,C_r$ transitively, by 
Proposition~\ref{prop:atlas}\ref{prop:atlasA7} we have only two 
possibilities:
\begin{enumerate}
\item 
$G_1\simeq \PSL_2(\bF_7)$, $r=15$;
\item 
$G_1\simeq \Alt_6$, $r=7$. 
\end{enumerate}
In both cases $\deg (C_1)\le 2$. Hence $C_1\simeq \PP^1$. 
On the other hand, the groups $\PSL_2(\bF_7)$ and $\Alt_6$ cannot act on $\PP^1$ effectively, a contradiction.

\section{Examples}
\label{sect:ex}
In this section we collect examples of $\Sym_6$-Fano threefolds.
Note that a faithful action of a group $G$ on an algebraic variety $X$
induces an embedding $G\subset \Bir(X)$ and an embedding $G\subset \Cr_n(\CC)$
if $X$ is rational.

\begin{example}
Let $\tilde \Sym_6$ be the pull-back of $\Sym _6\subset \operatorname{SO}_6(\RR)$ under the double cover 
\[
\operatorname{SU}_4(\CC)\to \operatorname{SO}_6(\RR). 
\]
Then $\tilde \Sym_6$ is a non-trivial central extension of $\Sym _6$
by $\mumu_2$. This defines an embedding
$\Sym_6\subset\PGL_4(\CC)$, so $\Sym _6$ acts on $\PP^3$.
\end{example}

In all examples below the group $\Sym_6$ is supposed to act on $x_1,\dots,x_6$ 
by permutations.

\begin{example} 
Let $X$ be a smooth quadric threefold given in $\PP^5$ by
\[
\sum_{i=1}^6 x_i=\sum_{i=1}^6 x_i^2=0.
\]
Then $X$ admits a $\Sym_6$-action by 
permutations of coordinates.
\end{example}

\begin{example}
\label{example-Segre-cubic}
The \textit{Segre cubic} $X_3^{\mathrm s}$ is a subvariety in $\PP^5$
given by the equations 
\[
\sum_{i=1}^6 x_i=\sum_{i=1}^6 x_i^3=0.
\]
The singular locus of this cubic consists of 10 nodes
and $\Aut X_3^{\mathrm s}\simeq \Sym_6$.
Moreover, the quotient $X_3^{\mathrm s}/\Sym_6\subset \PP^5/ \Sym_6$ is 
isomorphic to the weighted projective space
$\PP(2,4,5,6)\subset \PP(1,2,3,4,5,6)$. Therefore, $\rr(X_3^{\mathrm s})^{\Sym_6}=1$ and so
$X_3^{\mathrm s}$ is a $\Sym_6$-Fano threefold.
Obviously, $X_3^{\mathrm s}$ is rational. 
\end{example}

\begin{example}[{\cite[\S 4]{van-der-Geer-1982}}]
\label{example-quartics}
\label{example-quartic}
Let $X=X_4(\lambda)\subset \PP^4\subset \PP^5$ is a quartic given by the equations
\[
\sum_{i=1}^6 x_i=\sum_{i=1}^6 x_i^4+\lambda\Big(\sum_{i=1}^6 x_i^2\Big)^2=0,
\]
where $\lambda$ is a constant $\neq -1/4$.
The hypersurface $X_4(\lambda)$ is singular at the $30$ points of
the $\Sym_6$-orbit of $(1:\omega:\omega^2:1:\omega:\omega^2)$
with $\omega := e^{2\pi i/3}$. 
Additional isolated singularities occur only in the following cases
\begin{itemize}
\item 
$\lambda=-1/2$, $|\Sing(X)|=45$, 
\item
$\lambda=-7/10$, $|\Sing(X)|=36$,
\item
$\lambda=-1/6$, $|\Sing(X)|=40$. 
\end{itemize}
In the case $\lambda=-1/4$, the singularities are not isolated and
$X_4(\lambda)$ is so-called \textit{Igusa quartic}.

As in Example \xref{example-Segre-cubic} one can see that 
$\rr(X_4(\lambda))^{\Sym_6}=1$.
For $\lambda\notin \{-1/2, \, -7/10, \, -1/6\}$ the variety $X_4(\lambda)$ is not rational \cite{Beauville2013}.
For $\lambda=-1/2$ the variety $X_4(\lambda)$ is called \textit{Burkhardt quartic}. It is rational.
For $\lambda=-7/10$ and $-1/6$, $X_4(\lambda)$ is also rational \cite{Cheltsov-Shramov:S6}.
\end{example}

\begin{example}[double quadric]
\label{example-double-quadric}
Let 
$X=X_{2\cdot 4}\subset \PP(1^6,2)$ is given by 
\[
\sum_{i=1}^6 x_i=\sum_{i=1}^6 x_i^2=y^2-\sum_{i=1}^6 x_i^4=0,
\]
where $x_i$ and $y$ are coordinates in $\PP(1^6,2)$ with $\deg x_i=1$, $\deg y=2$.
This $X_{2\cdot 4}$ has $30$ nodes and $\rr(X_{2\cdot 4})=1$.
This variety was studied in \cite{PrzyjalkowskiShramov2016}. It is not rational.
\end{example}

\begin{example}[cubic complex]
\label{example-cubic-complex}
Consider the following complete intersections of 
a quadric and a cubic $X_6 \subset\PP^5$:
\begin{eqnarray*}
X_6(\lambda):& \sum_{i=1}^6 x_i^2=\sum_{i=1}^6 x_i^3-\lambda \Big(\sum_{i=1}^6 x_i\Big)^3=0,
\\
X_6':& 6\sum_{i=1}^6 x_i^2- \Big(\sum_{i=1}^6 x_i\Big)^2=\sum_{i=1}^6 x_i^3=0.
\end{eqnarray*}
Then these $X_6$'s are Fano threefolds with at worst Gorenstein terminal singularities.
Let $\lambda_0:=-1/18$, let $\lambda_1$ be a root of $180\lambda^2 + 20\lambda + 1$,
and let $\lambda_2$ be a root of $288\lambda^2 + 32\lambda + 1$.
Then 
\begin{enumerate}
\item 
$X_6'$ and $X_6(\lambda)$ for $\lambda\notin \{\lambda_0, \lambda_1,\bar \lambda_1, \lambda_2,\bar \lambda_2 \}$ are smooth,
\item 
$|\Sing(X(\lambda_1))|=|\Sing(X(\bar\lambda_1))|=6$,
\item 
$|\Sing(X(\lambda_2))|=|\Sing(X(\bar\lambda_2))|=15$,
\item 
$|\Sing(X(\lambda_0))|=20$.
\end{enumerate}
In all cases $\rr(X)^{\Sym _6}=1$. The general variety $X_6(\lambda)$ in the pencil is not rational \cite{Beauville2012}.
\end{example}

\section{Del Pezzo threefolds}
\label{section:del-Pezzo}
Recall that a \textit{del Pezzo variety} is a Fano variety $X$ with at worst 
canonical Gorenstein singularities such that the canonical class $K_X$
is divisible by $\dim(X)-1$ in the group $\Pic(X)$.
In this section, we consider only del Pezzo threefolds
with at worst terminal Gorenstein singularities.
In this case, $-K_X=2H$, where $H$ is am ample Cartier divisor.
Usually, $\PP^3$ is not considered as a del Pezzo threefold.
The self-intersection number $d:=H^3=-K_X^3/8$ is called the \textit{degree} of $X$. 

\begin{example}
\label{ex:flags}
Let $G=\Alt_6\subset \PGL_3(\CC)$ be the Valentiner group (see 
Lemma~\ref{rat-surfA6}).
Then $G$ acts faithfully on the variety $X_6$ of complete flags on $\PP^2$. 
This $X_6$ is a smooth del Pezzo variety of degree $6$ \cite{IP99, P:GFano1}. 
Note however that $\uprho(X_6)=2$ and the induced action of $\Alt_6$ on 
$\Pic(X_6)$ is trivial. Therefore, $X_6$ is not an $\Alt_6$-Fano threefold.
\end{example}

\begin{proposition}
\label{DP}
Assume that $\Alt_6$ faithfully acts on a del Pezzo threefold $X$
\textup(here we do not assume that $X$ is a $G$-Fano\textup).
Then $X$ is equivariantly isomorphic to one of the following varieties:
\begin{enumerate}
\item
\label{DP:segre}
$X=X_3\subset \PP^4$ is the Segre cubic \textup(see Example~\xref{example-Segre-cubic}\textup);

\item
\label{DP:flags}
$X$ is the variety of complete flags on $\PP^2$ \textup(see Example~\xref{ex:flags}\textup).
\end{enumerate}
\end{proposition}

\begin{proof}
Let $d$ be the degree of $X$.
It is easy to see by Riemann-Roch that 
\[
\dim H^0\big(X,-\textstyle{\frac12} K_X\big)=d+2.
\]
The action of the group $\Alt_6$ on $X$ lifts to an action of 
its central extension (double cover) $\tilde \Alt_6$ by $\mumu_2$ on $H^0(X,-\frac12 K_X)$.
Recall that there exists an exceptional isomorphism 
\[
\Alt_6\simeq \PSL_2(\bF_9)
\]
(see e.g. \cite[\S~3.3.5]{Wilson2009}). Thus a unique central extension $\tilde \Alt_6$ of $\Alt_6$ 
by $\mumu_2$ can be identified with $\SL_2(\bF_9)$. 
\begin{lemma}
\label{lemma:inv-h-s}
Let $X$ be a del Pezzo threefold that admit a faithful action of $\Alt_6$.
Then the linear system $|-\frac12 K_X|$ contains no invariant members.
\end{lemma}
\begin{proof}
Suppose that $S \in |-\frac12 K_X|$ is an invariant divisor.
Then $-(K_X+S)$ is ample, we can apply quite standard
connectedness arguments of Shokurov
(see, e.g., \cite[Lemma~B.5]{Prokhorov2014x}): 
for a suitable $G$-invariant boundary $D$, the pair 
$(X,D)$ is lc, the divisor $-(K_X+D)$ is ample,
and the minimal locus $V$ of log canonical singularities is also $G$-invariant. 
Moreover, $V$ is either a point or a smooth rational curve.
By Lemma \ref{lemma-fixed-point} the group $\Alt_6$ has no fixed points.
Hence, $\Alt_6\subset \Aut(\PP^1)$, a contradiction.
\end{proof}

The dimensions of irreducible representations of $\tilde \Alt_6=\SL_2(\bF_9)$ are as follows: $1$, $4$, $5$, $8$, $9$, $10$. 
By Lemma \ref{lemma:inv-h-s} $H^0(X,-\frac12 K_X)$ has no one-dimensional subrepresentations. Note that
$1\le d\le 7$ (see e.g. \cite{P:GFano1}). Therefore, $d\in \{2,\, 3,\, 6, \, 7\}$. 
Consider these possibilities one by one.

\par\medskip\noindent
\textbf{Case $d=2$.}
In this case the half-canonical map 
is a double cover $\pi: X\to \PP^3=\PP(H^0(X,-\frac12 K_X)^\vee)$ whose branch divisor 
$B\subset \PP^3$ is a quartic.

But it is easy to compute that the group $\tilde \Alt_6=\SL_2(\bF_9)$ has no invariants $0\neq \phi\in \operatorname{S}^4H^0(X,-\frac12 K_X)^\vee$.

\par\medskip\noindent
\textbf{Case $d=3$.} 
In this case $X=X_3\subset \PP^4$ is a cubic. The action of $\tilde \Alt_6$ on 
$H^0(X,-\frac12 K_X)$ is not faithful; it
is induced from the standard representation of $\Alt_6$ on $\CC^5$.
Then there is exactly one invariant hypersurface of degree $3$
(see \ref{DP}\ref{DP:segre}).

\par\medskip\noindent
\textbf{Case $d=6$.}
By \cite{P:GFano1} $X$ has at most one singular point.
Then by Lemma \ref{lemma-fixed-point} $X$
is smooth. Assume that $X\simeq \PP^1\times \PP^1\times \PP^1$.
Then the induced representation of $\Alt_6$ on $\Pic(X)$ is trivial.
Hence the group $\Alt_6$ effectively acts on the factors 
of $\PP^1\times \PP^1\times \PP^1$.
Clearly, this is impossible.
Therefore, $X\not \simeq \PP^1\times \PP^1\times \PP^1$.
Then by the classification \cite{IP99} $X$ is unique up to isomorphism, 
and it can be realized as the variety of complete flags on $\PP^2$. We get the case \ref{DP}\ref{DP:flags}.

\par\medskip\noindent
\textbf{Case $d=7$.}
In this case the variety $X$ is smooth and isomorphic to the blowup 
of a point on $\PP^3$ (see e.g. \cite{P:GFano1}). Thus the action of $\Alt_6$ descends to $\PP^3$.
On the other hand, this action has no fixed points by Lemma~\ref{lemma-fixed-point}, a contradiction.
\end{proof}

\begin{corollary}
\label{DP:S6}
Assume that the group $\Sym_6$ faithfully acts on a del Pezzo threefold $X$ of 
degree $d$.
Then $X$ is equivariantly isomorphic to the Segre cubic
\textup(Example~\xref{example-Segre-cubic}\textup).
\end{corollary}

\begin{proof}
The action of the subgroup $\Alt_6\subset \Sym_6$ is described in Proposition~\ref{DP}.
It is sufficient to show that the case \ref{DP}\ref{DP:flags} does not occur.
Assume that $X$ is isomorphic to the variety of complete flags on $\PP^2$.
In this case the linear system $|-\frac12 K_X|$ defines an embedding
$X=X_6\subset \PP^7$, where the space $\PP^7$ can be identified with the projectivization of
$H^0(X, -\frac12 K_X)^\vee$. The action of the group $\Sym_6$ on $X=X_6\subset \PP^7$ lifts to an action of 
a double cover $\tilde \Sym_6$ on $H^0(X,-\frac12 K_X)^\vee$.
By Lemma ~\ref{lemma:inv-h-s} the representation of $\tilde \Sym_6$ on $H^0(X,-\frac12 K_X)^\vee$ has no 
one-dimensional subrepresentations. Then $H^0(X,-\frac12 K_X)^\vee=W'\oplus 
W''$, where $W'$ and $W''$ are four-dimensional irreducible faithful 
representations. Thus $\PP(H^0(X, -\frac12 K_X)^\vee)$ contains two disjoint 
$\Sym_6$-invariant three-dimensional subspaces $\PP(W')$ and $\PP(W'')$.
Recall that $X=X_6\subset \PP^7$ is an intersection of quadrics.
Assume that $\PP(W')\cap X$ is not empty.
If $\dim(\PP(W')\cap X)=2$, then $\PP(W')\cap X$ is a $\Sym_6$-invariant quadric in 
$\PP(W')=\PP^3$. Clearly, this is impossible. 
If $\dim(\PP(W')\cap X)=1$, then $\PP(W')\cap X$ contains a $\Sym_6$-invariant
curve of degree $\le 4$ and the genus of this curve is at most $1$. 
Again  this is impossible.
Let $\dim(\PP(W')\cap X)=0$. Then $\PP(W')\cap X=\{P_1,\dots,P_k\}$, where 
$k\le 8$. By Corollary~\ref{cor:f-p} and Lemma~\ref{lemma-fixed-point}
the stabilizer $G_1\subset \Sym_6$ of $P_1$ is isomorphic to  $\Sym_5$. 
Note that the representation of $G_1$ in the 
tangent space $T_{P_1,X}$ is faithful. 
On the other hand, the group  $\Sym_5$ has no faithful three-dimensional representations, 
a contradiction.

Therefore, $\PP(W')\cap X=\emptyset$.
Then the projection $p:X \to \PP(W'')$ from $\PP(W')$ must be a $\Sym_6$-equivariant finite morphism. 
By the Hurwitz formula 
\[
K_{X} = p^* K_{\PP(W'')}- R
\]
where $R$ is the ramification divisor. Since $p^* K_{\PP(W'')}\sim -4H\sim 2 K_X$,
where $H$ is a hyperplane section of $X$, the divisor $R$ cannot be effective, a contradiction.
\end{proof}

\newcommand{\etalchar}[1]{$^{#1}$}
\def\cprime{$'$}


\end{document}